\documentclass[a4paper;12pt]{amsart}
\usepackage{amssymb,amsthm,amsmath}
\usepackage{latexsym,amssymb,amsmath,amsopn,graphics,xy,epsfig,picture,epic}
\usepackage{fullpage}
\usepackage{url}
\usepackage{graphicx}
\usepackage{geometry}

\theoremstyle{plain}
\newtheorem{theorem}{Theorem}
\newtheorem*{theorem*}{Theorem}
\newtheorem{nono-theorem}{Theorem}[]

\newtheorem{lemma}[theorem]{Lemma}

\theoremstyle{definition}

\newtheorem*{remark*}{Remark}

\newcommand{\RN}[1]{%
  \textup{\uppercase\expandafter{\romannumeral#1}}%
}

\input{tcilatex}

\begin{document}
\title{An estimate of the root mean square error incurred when approximating
an $f \in L^2({\mathbb{R}})$ by a partial sum of its Hermite series}
\author{Mei Ling Huang, Ron Kerman, Susanna Spektor}
\address{Mei Ling Huang \noindent Address: Department of Mathematics and Statistics, Brock University, St. Catharines,
Canada}
\email{mhuang@brocku.ca}
\address{Ron Kerman \noindent Address: Department of Mathematics and Statistics, Brock University, St. Catharines,
Canada}
\email{rkerman@brocku.ca}
\address{Susanna Spektor \noindent Address: Department of Mathematics and Statistics, Brock University, St.
Catharines, Canada}
\email{sanaspek@gmail.com}
\date{}
\date{}
\maketitle

\begin{abstract}
Let $f$ be a band-limited function in $L^2({\mathbb{R}})$. Fix $T >0$ and
suppose $f^{\prime}$ exists and is integrable on $[-T, T]$. This paper gives
a concrete estimate of the error incurred when approximating $f$ in the root
mean square by a partial sum of its Hermite series.

Specifically, we show, for $K=2n, \quad n \in Z_+,$
\begin{align*}
\left[\frac{1}{2T}\int_{-T}^T[f(t)-(S_Kf)(t)]^2dt\right]^{1/2}\leq \left(1+\frac 1K\right)\left(\left[
\frac{1}{2T}\int_{|t|> T}f(t)^2dt\right]^{1/2} +\left[\frac{1}{2T}
\int_{|\omega|>N}|\hat f(\omega)|^2d\omega\right]^{1/2} \right)\\
+\frac{1}{K}\left[\frac{1}{2T}\int_{|t|\leq T}f_N(t)^2dt\right]^{1/2} +\frac{1}{\pi}\left(1+\frac{1}{2K}\right)S_a(K,T),
\end{align*}
in which $S_Kf$ is the $K$-th partial sum of the Hermite series of $f, \hat
f $ is the Fourier transform of $f$, $\displaystyle{N=\frac{\sqrt{2K+1}+%
\sqrt{2K+3}}{2}}$ and $f_N=(\hat f \chi_{(-N,N)})^\vee(t)=\frac{1}{\pi}\int_{-\infty}^{\infty}\frac{\sin (N(t-s))}{t-s}f(s)ds$. An explicit upper bound is obtained for $S_{a}(K,T)$.
\end{abstract}

\medskip



\section{Introduction}

We recall that the $k$-th Hermite function, $h_k$, is given at $t \in {%
\mathbb{R}}$ by
\begin{equation*}
h_k(t)=(-1)^{k}\gamma_ke^{\frac{t^2}{2}}\frac{d^ke^{-t^2}}{dt^k}, \quad
k=0,1, \ldots,
\end{equation*}
where $\gamma_k=\pi^{-1/2}2^{-k/2}(k!)^{-1/2}$. Given $f \in L^2({\mathbb{R}}%
)$, its Hermite series is
\begin{equation*}
\sum_{k=0}^{\infty}c_kh_k,
\end{equation*}
in which
\begin{equation*}
c_k=\int_{{\mathbb{R}}}f(t)h_k(t)dt, \quad k=0,1,2,\ldots.
\end{equation*}

We seek an estimate in the root mean square of how well the $K$-th partial
sum of the Hermite series of $f$, namely,
\begin{equation*}
(S_Kf)(t)=\sum_{k=0}^K c_kh_k(t),
\end{equation*}
approximates it. Our principal result is

\begin{theorem}
{\label{th1}} Consider a band-limited function $f\in L^2({\mathbb{R}})$. Fix $T>0$
and suppose $f^{\prime}$ exists and is integrable on $I_T=[-T,T]$.

Then, with $K=2n, \quad n\in Z_+$, and
$S_Kf$ the $K$-th partial sum of the Hermite series of $f$, one has
\begin{equation*}
N=\frac{\sqrt{2K+1}+\sqrt{2K+3}}{2}
\end{equation*}
and $S_Kf$ the $K$-th partial sum of the Hermite series of $f$, one has
\begin{align}  \label{eq1}
\left[\frac{1}{2T}\int_{-T}^T[f(t)-(S_Kf)(t)]^2dt\right]^{1/2}\leq \left(1+\frac 1K\right)\left(\left[%
\frac{1}{2T}\int_{|t|> T}f(t)^2dt\right]^{1/2} +\left[\frac{1}{2T}%
\int_{|\omega|>N}|\hat f(\omega)|^2d\omega\right]^{1/2} \right)\\
+\frac{1}{K}\left[\frac{1}{2T}\int_{|t|\leq T}f_N(t)^2dt\right]^{1/2} +\frac{1%
}{\pi}\left(1+\frac{1}{2K}\right)S_a(K,T),
\end{align}
in which $f_N=(\hat f \chi_{(-N,N)})^\vee$.
\end{theorem}

The goal now is to find, for an appropriate $T$, the smallest $K$ to ensure
the right hand side of (\ref{eq1}) satisfies a given bound. To do this we
need explicit bounds for $S_a(K, T)$. We will make careful use of the
estimates of the kernel of an integral representation of $S_Kf$ due to
Sansone; see [S]. These estimates show that the core of the partial sum
operator is the Dirichlet operator, $F_N$, defined at $f$ on $I_T$ by
\begin{equation*}
(F_Nf)(t)=\frac{1}{\pi}\int_{-T}^T\frac{\sin(N(t-s))}{t-s}f(s)ds, \quad t
\in I_T.
\end{equation*}

A key fact, used repeatedly in the derivation at our estimates, is that the
Hilbert transform, $H$, given, for suitable $f$ at almost all $x \in {%
\mathbb{R}}$ by
\begin{equation*}
(Hf)(x)=\frac{1}{\pi}(P)\int_{{\mathbb{R}}}\frac{f(y)}{x-y}dy=\frac{1}{\pi}%
\lim_{\varepsilon \to 0^+}\int_{|x-y|>\varepsilon}\frac{f(y)}{x-y}dy,
\end{equation*}
is a unitary operator on $L^2({\mathbb{R}})$. Also important will be certain
identities of Bedrosian, valid for band-limited fucntions $f \in L^2({%
\mathbb{R}})$, namely,
\begin{equation*}
H(f \sin(a \cdot))(t)=f(t) \cos at
\end{equation*}
and
\begin{equation*}
H(f \cos(b \cdot))(t)=f(t) \sin bt, \quad t\in {\mathbb{R}},
\end{equation*}
for fixed $a,b \in {\mathbb{R}}$. See [B].

The error involved in approximating $f$ by $F_Nf$ is established in Lemma 2
in the next section. The Sansone estimates are intensively studied in
Section 3. These enable the proof of Theorem 1 in the following section,
where  $S_a(n,T)$ is defined. An explicit estimate of $S_a(n,T)$ is described in an appendix.

The estimate of the root mean square error in (\ref{eq1}) is both more
specific and more easily calculated than the one in the paper [KHB] of the
first two authors and M.~Brannan. In the final section we revisit the
trimodal distribution studied in that paper.


\section{Approximation using the Dirichlet operator}

In this section we prove

\begin{lemma}
\label{lem 2} Given $f \in L^2({\mathbb{R}})$, $N, T \in {\mathbb{R}}_+$,
and $f_T=f \chi_{(-T,T)}$, set

\begin{equation*}
(F_Nf_T)(t)=\frac{1}{\pi} \int_{-T}^T \frac{\sin( N(t-s))}{t-s}f(s) ds, \quad t \in I_T=[-T,T].
\end{equation*}
Then,
\begin{equation*}
\left[\int_{-1}^1|f(t)-(F_Nf_T)(t)|^2 dt\right]^{1/2}dt\leq \left[%
\int_{|t|>T}|f(t)|^2 dt\right]^{1/2}+\left[\int_{|\omega|> N}|\hat{f}%
(\omega)|^2 d \omega\right]^{1/2}\renewcommand{\thefootnote}{(1)}\footnote{%
We define the Fourier transform, $\hat{f}$, of $f$ by
\begin{equation*}
\hat{f}(\omega)=\frac{1}{\sqrt{2 \pi}}\int_{{\mathbb{R}}} f(t)e^{-i\omega t}
dt, \quad \omega \in {\mathbb{R}}.
\end{equation*}%
.}
\end{equation*}
\end{lemma}

\begin{proof}
Using the standard notation $\displaystyle{sinc t=\frac{\sin t}{t}}$, we
have
\begin{equation*}
F_N f_T=\sqrt{\frac{2}{\pi}} N(sinc\, t)\ast f_T\renewcommand{%
\thefootnote}{(2)}\footnote{%
The convolution, $g \ast h$, of $g$ and $h$ in $L^2({\mathbb{R}})$ is here
defined by $\displaystyle{(g \ast h)(t)=\frac{1}{\sqrt{2 \pi}}%
\int_{-\infty}^{\infty}g(t-s)h(s)\, ds}$, $t \in {\mathbb{R}}$. One has $(g
\ast h)^{\wedge}(\omega)={g^{\wedge}}(\omega){h^{\wedge}}(\omega)$, $\omega
\in {\mathbb{R}}$. See \cite{Wiener}.} , \quad t\in {\mathbb{R}}.
\end{equation*}
whence
\begin{equation*}
\widehat{F_Nf_T}=\chi_{(-N, N)}\hat{f}_T.
\end{equation*}
Thus,
\begin{align*}
&\left[\int_{-T}^T |f(t)-(F_N f_T)(t)|^2 dt\right]^{1/2}\leq \left[\int_{{%
\mathbb{R}}}|f(t)-(F_Nf_T)(t)|^2 dt\right]^{1/2}
=\left[\int_{{\mathbb{R}}_+}|\hat{f}(\omega)-\widehat{F_Nf_T}(\omega)|^2 d
\omega\right]^{1/2} \\
&=\left[\int_{{\mathbb{R}}}|\hat{f_T}(\omega)-\chi_{(-N, N)}\widehat{f}%
(\omega)_T|^2 d \omega\right]^{1/2}
=\left[\int_{|\omega|>N}|\hat{f}_T(\omega)|^2 d \omega\right]^{1/2}
=\left[\int_{|\omega|>N}\left(\frac{1}{\sqrt{2 \pi}}\int_{-T}^T f(y)e^{-i{%
\omega}y}dy\right)^2 d \omega\right]^{1/2} \\
&\leq\left[\int_{|\omega|>N}\left|\frac{1}{\sqrt{2 \pi}}\int_{{\mathbb{R}}%
}\chi_{T}(y) f(y)e^{-i\omega y}dy\right|^2 d \xi \right]^{1/2}
=\left[\int_{|\omega|>N}\left((\widehat{\chi_T f}(\omega)\right)^2 d \omega%
\right]^{1/2} \\
&\leq\left[\int_{|\omega|>N}\left|\hat{f}(\omega)\right|^2 d \omega\right]%
^{1/2}+\left[\int_{|\omega|>N}\left|(\chi_{|t|>T}f)^{\hat{}%
}(\omega)\right|^2 d \omega\right]^{1/2}
\leq \left[\int_{|\omega|>N}\left|\hat{f}(\omega)\right|^2 d \omega\right]%
^{1/2}+\left[\int_{{\mathbb{R}}}\left|(\chi_{|t|>T}f)^{\hat{}%
}(\omega)\right|^2 d \omega\right]^{1/2} \\
&\leq \left[\int_{|\omega|>N}\left|\hat{f}(\omega)\right|^2 d \omega\right]%
^{1/2}+\left[\int_{|t|>T}\left|{f}(t)\right|^2 dt \right]^{1/2}.
\end{align*}
\end{proof}

\section{The Sansone estimates}

To begin, we describe Sansone's analysis of the usual expression for $%
(\sum\nolimits_{K}f_T)(t)=\sum_{k=0}^K c_kh_k(t)$, when $K$ is even, say $%
K=2n$. For ease of reference to \cite{Uspensky:1927} we work with the
variables $x$ and $\alpha$ rather than $t$ and $s$.

Now, according to [S, p. 372, (4) and (5)],
\begin{align}  \label{2.1}
(S_{2n}f)(x)=\sqrt{\frac{2n+1}{2}}\int_{-T}^T k_{2n}(x, \alpha)f(\alpha) d
\alpha,
\end{align}
where
\begin{equation*}
k_{2n}(x, \alpha) =\frac{h_{2n+1}(x)h_{2n}(\alpha)-h_{2n+1}(\alpha)h_{2n}(x)%
}{\alpha -x}
\end{equation*}
and $\displaystyle{f_T=f_T\chi_I}$. Further, by (7) and (8) on p. 373 and
the first two estimates on p. 374, together with (15.1) and (15.2) on p.
362, one has
\begin{align}  \label{2}
\sqrt{\frac{2k+1}{2}}k_{2n}(x,
\alpha)(x-\alpha)=-C^{(n)}[A(x)B(\alpha)-A(\alpha)B(x)],
\end{align}
in which
\begin{equation*}
-C^{(n)}=\frac{1}{\pi}\left(1+\frac{\varepsilon}{12 n}\right), \quad
|\varepsilon|<3,
\end{equation*}
\begin{equation*}
A(y)=\sin(\sqrt{4n+3}y)-\frac{y^3}{6}\frac{\cos(\sqrt{4n+3}y)}{\sqrt{4n+3}}+%
\frac{T(2n+1, y)}{h_{2n+1}^{\prime}(0)\sqrt{4n+3}}
\end{equation*}
and
\begin{equation*}
B(y)=\cos(\sqrt{4n+1}y)-\frac{y^3}{6}\frac{\sin(\sqrt{4n+1}y)}{\sqrt{4n+1}}+%
\frac{T(2n, \alpha)}{h_{2n}(0)\sqrt{4n+1}}.
\end{equation*}
The functions $T(2n, y)$ and $T(2n+1, y)$ are defined through the equations
\begin{equation*}
h_{2n}(x)=h_{2n}(0)\cos(\sqrt{4n+1}x)+\frac{h_{2n}(0)}{\sqrt{4n+1}}\frac{y^3%
}{6} \sin(\sqrt{4n+1}y)+\frac{T(2n, y)}{4n+1}
\end{equation*}
and
\begin{equation*}
h^{\prime}_{2n+1}(y)=h^{\prime}_{2n+1}(0)\frac{\sin(\sqrt{4n+3} y)}{\sqrt{%
4n+3}}-\frac{h^{\prime}_{2n+1}(0)}{4n+3}\frac{y^3}{6}\cos(\sqrt{4n+3}y)+%
\frac{T(2n+1, y)}{4n+3}.
\end{equation*}

Here,
\begin{equation*}
|a_n|=\frac{1}{|h^{\prime}_{2n+1}(0)|}\frac{1}{\sqrt{4n+3}}< \frac{4}{%
\pi^{1/2}}\sqrt[4]{\frac 32} \frac{1}{n^{3/4}},
\end{equation*}
\begin{equation*}
|b_n|=\frac{1}{|h_{2n}(0)|}\frac{1}{4n+1}< \frac{\pi^{1/2}}{4}\sqrt[4]{\frac
32}\frac{1}{n^{3/4}},
\end{equation*}
\begin{equation*}
|T(2n+1, y)|< \frac{y^2}{\pi^{1/2}n^{1/4}}\left(\frac{y^4}{18}+1\right)+%
\frac{4}{187} \frac{y^{17/2}}{\sqrt{4n+1}}
\end{equation*}
and
\begin{equation*}
|T(2n, y)|< \frac{y^2}{\pi^{1/2}n^{1/4}}\left(\frac{y^4}{18}+1\right)+\frac{4%
}{187} \frac{y^{17/2}}{\sqrt{4n+3}}.
\end{equation*}

Expanding the products in (\ref{2}) yields
\begin{align}  \label{3}
\sqrt{\frac{2n+1}{2}}k_{2n}(x, \alpha)(x-
\alpha)=-\frac{1}{\pi}\left(1+\frac{\varepsilon}{6K}\right)\left(\sin(N(x-\alpha))+\sum_{k=1}^5M_k^{(n)}(x, \alpha)\right),\quad |\varepsilon|<3,
\end{align}
where, firstly,
\begin{equation*}
M_1^{(n)}(x, \alpha)=\cos(N(x+ \alpha)){\sin((x-\alpha)/2N)}-2\sin^2((x+\alpha)/4N)\sin(N(x-\alpha)),
\end{equation*}
as shown on pp. 375 of \cite{Uspensky:1927}. Again, on p. 376 we find
\begin{align*}
&\sqrt{4n+1}M_2^{(n)}(x, \alpha)=\frac{-x^3}{6}\sin(\sqrt{4n+1}x)\sin(\sqrt{%
4n+3}\alpha)
+\frac{\alpha^3}{6}\sin(\sqrt{4n+1}\alpha) \sin(\sqrt{4n+3} x) \\
&=\frac{\alpha^3-x^3}{6} \sin(\sqrt{4n+1}x) \sin(\sqrt{4n+3} \alpha)
+ \frac{\alpha^3}{6}\left[\sin((x
-\alpha)/2N)\sin(N(x+\alpha))-\sin((x+\alpha)/2N)\sin(N(x-\alpha))\right].
\end{align*}
An argument similar to the one for $\sqrt{4n+1}M_2^{(n)}(x, \alpha)$ gives
\begin{align*}
&\sqrt{4n+3}M_3^{(n)}(x, \alpha)=\frac{\alpha^3}{6}\cos(\sqrt{4n+1}x)\cos(%
\sqrt{4n+3}\alpha)
-\frac{x^3}{6} \cos (\sqrt{4n+1}\alpha)\cos(\sqrt{4n+3}x) \\
&=\frac{\alpha^3-x^3}{6} \cos(\sqrt{4n+1}x)\cos(\sqrt{4n+3}x)
-\frac{x^3}{6} \left[\cos (\sqrt{4n+1}\alpha)\cos(\sqrt{4n+3}x)- \cos (%
\sqrt{4n+1}x)\cos(\sqrt{4n+3}\alpha)\right] \\
&=\frac{\alpha^3-x^3}{6} \cos(\sqrt{4n+1}x)\cos(\sqrt{4n+3}\alpha)
-\frac{x^3}{6} [\sin(N(x +\alpha))\sin((x-\alpha)/2N)+ \sin
((x+\alpha)/2N)\sin(N(x-\alpha))].
\end{align*}

Next, \bigskip
\begin{equation*}
\hspace{-8cm}\sqrt{(4n+1)(4n+3)}M_{4}^{(n)}(x, \alpha)
\end{equation*}
\begin{equation*}
\hspace{3cm}=\frac{\alpha^3 x^3}{36}[-\cos(\sqrt{4n+3}x)\sin(\sqrt{4n+1}%
\alpha)+\cos(\sqrt{4n+3}\alpha)\sin(\sqrt{4n+1}x)]\renewcommand{%
\thefootnote}{(3)}\footnote{%
The expression for $\sqrt{(4n+1)(4n+3)}M_{4}^{(n)}(x, \alpha)$ on p. 345--372 of
[S] is incorrect.}
\end{equation*}
\begin{align*}
\hspace{3cm}&=\frac{\alpha^3x^3}{72}\Big[\sin(\sqrt{4n+3}x-\sqrt{4n+1}%
\alpha)-\sin(\sqrt{4n+3}x+\sqrt{4n+1} \alpha) \\
\hspace{3cm}&+\sin(\sqrt{4n+1}x-\sqrt{4n+3}\alpha)+\sin(\sqrt{4n+1}x+\sqrt{%
4n+3}\alpha)\Big] \\
&=\frac{\alpha^3 x^3}{36}[\sin(N(x -\alpha))\cos((x+\alpha)/2N)- \sin
((x-\alpha)/2N)\cos(N(x+\alpha))].
\end{align*}

Finally,
\begin{align*}
M_5^{(n)}(x, \alpha)&=a_n\Big[T(2n+1, x)\cos(\sqrt{4n+1} \alpha)-T(2n+1,
\alpha)\cos(\sqrt{4n+1}x) \\
&+T(2n+1, x)\frac{\alpha^3}{6}\frac{\sin(\sqrt{4n+1}\alpha)}{\sqrt{4n+1}}%
-T(2n+1, \alpha)\frac{x^3}{6}\frac{\sin(\sqrt{4n+1}x)}{\sqrt{4n+1}}\Big] \\
&+b_n\Big[T(2n, x)\sin(\sqrt{4n+3} \alpha)-T(2n, \alpha)\sin(\sqrt{4n+3}x) \\
&+T(2n, \alpha)\frac{x^3}{6}\frac{\cos(\sqrt{4n+3}x)}{\sqrt{4n+3}}-T(2n, x)%
\frac{\alpha^3}{6}\frac{\cos(\sqrt{4n+3}\alpha)}{\sqrt{4n+3}}\Big] \\
&+a_nb_n[T(2n+1, x)T(2n, \alpha)-T(2n+1, \alpha)T(2n, x)],
\end{align*}

To prove Theorem 1 we will require the following estimates of integrals
involving terms on the right hand side of (\ref{2.1}).
\bigskip

3.1
\begin{align*}
&\left|\int_{-T}^TM_1^{(n)}(x, \alpha)\frac{f(\alpha)}{x-\alpha}d\alpha\right|
\leq \frac{1}{2N}\left|\int_{-T}^T\cos(N(x+\alpha))\sin\left({x-\alpha}/{2N}\right)f(\alpha)d\alpha\right|\\\
&+2\left|\int_{-T}^T \sin^2((x+\alpha)/4N)\sin(N(x-\alpha))\frac{f(\alpha)}{x-\alpha}d\alpha\right|
=\RN{1}(x)+\RN{2}(x).
\end{align*}
To begin,
\begin{align*}
\RN{1}(x)\leq &\frac{1}{2N}\left[\left|\int_{-T}^T\cos(N(x+\alpha))\left[\frac{%
\sin((x-\alpha)/2N)}{(x-\alpha)/2N}-1\right]f(\alpha)\, d\alpha\right|
+\left|\int_{-T}^T\cos(N(x+\alpha))f(\alpha)\, d\alpha\right|\right] \\
&\leq \frac{1}{2N}\left[\int_{-T}^T \frac 16 \left(\frac{x-\alpha}{2N}%
\right)^2|f(\alpha)|d\alpha +\frac{|f(-T)|+|f(T)|}{N}+\frac
1N\int_{-T}^T|f^{(1)}(\alpha)|d \alpha\right] \\
&\leq \frac{1}{48N^3}\left[x^2\int_{-T}^T|f(\alpha)|d\alpha+2|x|%
\int_{-T}^T|f(\alpha)\alpha|d\alpha +\int_{-T}^T|f(\alpha)\alpha^2|d\alpha %
\right] +\frac{1}{2N^{2}}\left[|f(-T)|+|f(T)|+\int_{-T}^T|f^{(1)}(\alpha)|d%
\alpha\right].
\end{align*}
Thus,
\begin{align*}
&\left[\frac{1}{2T}\int_{-T}^TI(x)^2dx\right]^{1/2}\leq \frac{1}{48N^3}%
\Big[\left[\frac{1}{2T}\int_{-T}^Tx^4dx\right]^{1/2}\int_{-T}^T|f(\alpha)|d%
\alpha \\
&+2\left[\frac{1}{2T}\int_{-T}^Tx^2dx\right]^{1/2}\int_{-T}^T|f(\alpha)%
\alpha|d\alpha +\int_{-T}^T|f(\alpha)\alpha^2|d\alpha\Big]+\frac{1}{2N^{2}}%
\left[|f(-T)|+|f(T)|+\int_{-T}^T|f^{(1)}(\alpha)|d\alpha\right] \\
&=\frac{1}{48 N^3}\left[\frac{T^2}{\sqrt 5}\int_{-T}^T|f(\alpha)|d\alpha+%
\frac{T}{\sqrt 3}\int_{-T}^T|f(\alpha)\alpha|d\alpha
+\int_{-T}^T|f(\alpha)\alpha^2|d\alpha\right] + \frac{1}{2N^{2}}\left[%
|f(-T)|+|f(T)|+\int_{-T}^T|f^{(1)}(\alpha)|d\alpha \right].
\end{align*}

Again,
\begin{align*}
\RN{2}(x)&=\left|\int_{-T}^T(1-\cos((x+\alpha)/2N))\sin(N(x-\alpha))\frac{f(\alpha)}{x-\alpha}d\alpha\right|
\leq \left|\int_{-T}^T\frac{((x+\alpha)/2N)^2}{2}\sin(N(x-\alpha))\frac{f(\alpha)}{x-\alpha}d\alpha\right|\\
&\leq \frac{1}{8N^2}\Big[x^2\left|\int_{-T}^T
sin(N(x-\alpha))\frac{f(\alpha)}{x-\alpha} d\alpha\right|
+2|x|\left|\int_{-T}^T {\sin(N(x-\alpha))}\frac{f(\alpha)\alpha}{(x-\alpha)}\, d\alpha\right|\\
&+\left|\int_{-T}^T {\sin(N(x-\alpha))}\frac{f(\alpha)\alpha^2}{x-\alpha}\, d\alpha\right|
\Big]
=\frac{1}{8N^2}[\RN{3}(x)+\RN{4}(x)+\RN{5}(x)].
\end{align*}

Now,
\begin{align*}
\RN{4}(x)&\leq 2T \left|\sin(Nx)\int_{-T}^T\frac{f(\alpha)\alpha \cos(N\alpha)}{x-\alpha}d\alpha-\cos(Nx)\int_{-T}^T\frac{f(\alpha)\alpha \sin(N\alpha)}{x-\alpha}d\alpha\right|,
\end{align*}
so,
\begin{align*}
\left[ \frac{1}{2T} \int_{-T}^T \RN{4}(x)^2 dx\right]^{1/2}&\leq 2 \pi T^{1/2}\left[\int_{-T}^T |f(\alpha)\alpha|^2 d\alpha\right]^{1/2}.
\end{align*}

Similarly,
\begin{align*}
\left[ \frac{1}{2T} \int_{-T}^T \RN{5}(x)^2 dx\right]^{1/2}&\leq \sqrt 2 \pi T^{-1/2}\left[\int_{-T}^T |f(\alpha)\alpha^2|^2 d\alpha\right]^{1/2}.
\end{align*}

Next,
\begin{equation*}
\left|\int_{-T}^T\sin(N(x-\alpha))\frac{f(\alpha)}{x-\alpha}d\alpha\right|\leq \left|\int_{-T}^T \frac{f(\alpha)\sin(N\alpha)}{x-\alpha}\right|+\left|\int_{-T}^T \frac{f(\alpha)\cos(N\alpha)}{x-\alpha}\right|
\end{equation*}
and, by Bedrosian's identity,
\begin{align*}
&\int_{-T}^T\frac{f(\alpha)\sin(N\alpha)}{x-\alpha}dx=\int_{-\infty}^{\infty}\frac{f(\alpha)\sin(N\alpha)}{x-\alpha}d\alpha-\int_{-\infty}^{\infty}\frac{f(\alpha)\sin(N\alpha)}{x-\alpha} \chi_{|\alpha|>T}d\alpha\\
&=\int_{-\infty}^{\infty}\frac{f_N(\alpha)\sin(N\alpha)}{x-\alpha}d\alpha+\int_{-\infty}^{\infty}\frac{[f(\alpha)-f_N(\alpha)]\sin(N\alpha)}{x-\alpha} d\alpha-\int_{-\infty}^{\infty}\frac{[f(\alpha)-f_N(\alpha)]\sin(N\alpha)}{x-\alpha} \chi_{|\alpha|>T}d\alpha,
\end{align*}
where $f_N=(\hat f\chi_{(-N,N)})^{\vee}$.

A similar result holds with $\sin(N\alpha)$ replaced by $\sin(N\alpha)$.


Thus,

\begin{align*}
\left[\frac{1}{2T}\int_{-T}^T \RN{3}^2 dx\right]^{1/2}&\leq \sqrt 2 \pi T^{3/2}\left(\left[\int_{-T}^Tf_N(\alpha)^2 d\alpha\right]^{1/2}+\left[\int_{|\alpha|>T}f(\alpha)^2 d\alpha\right]^{1/2}+\left[\int_{|\omega|>N}|\hat f(\omega)|^2d\omega\right]^{1/2}\right),
\end{align*}
since
\begin{align*}
\int_{-\infty}^{\infty}|f(\alpha)-f_N(\alpha)|^2 d\alpha=\int_{-\infty}^{\infty}|\hat f(\omega)-\hat{f}_N(\omega)|^2d\omega=\int_{|\omega|>N}|\hat f(\omega)|^2 d\omega.
\end{align*}
In sum,
\begin{align*}
&\left[\frac{1}{2T}\int_{-T}^T\left|\int_{-T}^TM_1^{(n)}(x, \alpha)\frac{f(\alpha)}{x-\alpha}d\alpha\right|^2dx\right]^{1/2}
\leq\frac{1}{48N^3}\left[\frac{T^2}{\sqrt 5}\int_{-T}^T|f(\alpha)|d\alpha+\frac{T}{3}\int_{-T}^T|f(\alpha)\alpha|d\alpha +\int_{-T}^T|f(\alpha)|\alpha^2d\alpha\right]\\
&+\frac{1}{2N^2}\left[|f(-T)|+|f(T)|+\int_{-T}^T|f^{(1)}(\alpha)|d\alpha\right]
+\frac{1}{8N^2}\Bigg(2\pi T^{1/2}\left[\int_{-T}^T|f(\alpha)\alpha|^2 d\alpha\right]^{1/2}+\sqrt2 \pi T^{-1/2} \left[\int_{-T}^T|f(\alpha)\alpha^2|^2 d\alpha\right]^{1/2}\\
&+\sqrt 2 \pi T^{3/2}\left(\left[\int_{-T}^Tf_N(\alpha)^2 d\alpha\right]^{1/2}+\left[\int_{|\alpha|>T}|f(\alpha)|^2 d\alpha\right]^{1/2}+\left[\int_{|\omega|>N}|\hat f(\omega)|^2 d\omega\right]^{1/2}\right)\Bigg)
\end{align*}

\bigskip

3.2
\begin{align*}
&\left|\int_{-T}^T \frac{M_n^{(2)}(x, \alpha)}{x-\alpha}f(x) d
\alpha\right|\leq \frac{1}{\sqrt{4n+1}}\left|\int_{-T}^T \frac{\alpha^3-x^3%
}{6(x- \alpha)}\sin(\sqrt{4n+3} \alpha)f(\alpha) d\alpha\right|
+\frac{1}{\sqrt{4n+1}} \left|\int_{-T}^T {sinc((x- \alpha)/2N)}%
sin(N(x+\alpha))f(\alpha) \frac{\alpha^3}{12N} d \alpha\right| \\
&+\frac{1}{\sqrt{4n+1}} \left|\int_{-T}^T {sinc((x+ \alpha)/2N)}sinc
(N(x-\alpha))f(\alpha) \frac{(x+\alpha)\alpha^3}{12} d \alpha\right|.
\end{align*}

Now,
\begin{align*}
&\RN{1}(x)\leq \frac{1}{6\sqrt{4n+1}}\left[\left|\int_{-T}^T \sin(\sqrt{4n+3}
\alpha)f(\alpha)\alpha^{2}\, d \alpha\right|+|x|\left|\int_{-T}^T\sin(\sqrt{%
4n+3}\alpha) f(\alpha)\alpha d\alpha\right|+x^2\left|\int_{-T}^T\sin(\sqrt{%
4n+3}\alpha)f(\alpha)d\alpha\right|\right] \\
&\leq \frac{1}{6\sqrt{(4n+1)(4n+3)}}\Big[T^2(|f(-T)|+|f(T)|)+\int_{-T}^T%
\left|\frac{d}{d\alpha}(f(\alpha)\alpha^2)\right|d\alpha +|x|(T(|f(-T)|+|f(T)|)\\
&+\int_{-T}^T\left|\frac{d}{\alpha}(f(\alpha)\alpha)%
\right|d\alpha) x^2\left(|f(-T)|+|f(T)|+\int_{-T}^T|f^{(1)}(\alpha)|d\alpha\right) \Big] \\
&\leq \frac{1}{6\sqrt{(4n+1)(4n+3)}}\Big[(x^2+|x|T+T^2)(|f(-T)|+|f(T)|)+%
\int_{-T}^T|f^{(1)}(\alpha)\alpha^2|d\alpha \Big] +2\int_{-T}^T|f(\alpha)\alpha|d\alpha \\
&+|x|\left(\int_{T}^T|f^{(1)}(\alpha)\alpha|d\alpha+\int|f(\alpha)|d\alpha%
\right) +x^2\int_{-T}^T|f^{(1)}|d\alpha
\end{align*}
Thus,
\begin{align*}
&\left[\frac{1}{2T}\int_{-T}^T \RN{1}(x)^2dx\right]^{1/2}\leq \frac{1}{6\sqrt{%
(4n+1)(4n+3)}}\Big[(1+\frac{1}{\sqrt 3}+\frac{1}{\sqrt 5})T^2(|f(-T)|+|f(T)|)
\\
&+\frac{T^2}{\sqrt 5}\int_{-T}^T \left|f^{(1)}(\alpha)\right| d \alpha+\frac{%
T}{\sqrt 3}\left(\int_{-T}^T\left| f^{(1)}(\alpha)\alpha \right|
d\alpha+\int_{-T}^T\left|f(\alpha)\right|d\alpha\right) +\int_{-T}^T\left|f^{(1)}(\alpha)\alpha^2\right|d\alpha+2\int_{-T}^T%
\left|f(\alpha)\alpha\right|d\alpha\Big].
\end{align*}
Next,
\begin{align*}
\RN{2}(x)&\leq \frac{1}{12N\sqrt{4n+1}}\left[\int_{-T}^T|\sin((x-%
\alpha)/2N)-1||f(\alpha)\alpha^3|d\alpha+\left|\int_{-T}^T
\sin(N(x+\alpha)f(\alpha)\alpha^3d\alpha\right|\right] \\
&\leq \frac{1}{12N\sqrt{4n+1}}\left[\int_{-T}^T\frac 16 \left(\frac{x-\alpha%
}{2N}\right)^3|f(\alpha)\alpha^3|d\alpha+\frac{2}{N}\int_{-T}^T \left|\frac{%
d(f(\alpha)\alpha^3)}{d\alpha}\right| d\alpha\right] \\
&\leq \frac{1}{288N^3\sqrt{4n+1}}\left[x^2\int_{-T}^T|f(\alpha)\alpha^3|d%
\alpha+2|x|\int_{-T}^T |f(\alpha)\alpha^4|d\alpha+\int_{-T}^T
|f(\alpha)\alpha^5|d\alpha\right] \\
&+\frac{1}{6N^{2}\sqrt{4n+1}}\left[T^3(|f(-T)|+|f(T)|)+\int_{-T}^T|f^{(1)}(%
\alpha)\alpha^3|d\alpha+3\int_{-T}^T |f(\alpha)\alpha^2|d\alpha\right].
\end{align*}
Hence,
\begin{align*}
\left[\frac{1}{2T}\int_{-T}^T\RN{2}(x)^2dx\right]^{1/2}& \leq \frac{1}{288N^3%
\sqrt{4n+1}}\left[\frac{T^2}{\sqrt 5}\int_{-T}^T|f(\alpha)\alpha^3|d\alpha+%
\frac{2T}{\sqrt 5}\int_{-T}^T |f(\alpha)\alpha^4|d\alpha+\int_{-T}^T
|f(\alpha)\alpha^5|d\alpha\right] \\
&+\frac{1}{6N^{2}\sqrt{4n+1}}\left[T^3(|f(-T)|+|f(T)|)+\int_{-T}^T|f^{(1)}(%
\alpha)\alpha^3|d\alpha+3\int_{-T}^T |f(\alpha)\alpha^2|d\alpha\right].
\end{align*}

Finally,
\begin{align*}
\RN{3}(x)&\leq \frac{1}{48N^2\sqrt{4n+1}}\Big[|x|^3\int_{-T}^T
|f(\alpha)\alpha^3|d \alpha
+3x^2\int_{-T}^T |f(\alpha)\alpha^4|d\alpha+\left|\int_{-T}^T
sinc(N(x-\alpha))f(\alpha)\alpha^4d\alpha\right|\Big] \\
&\leq \frac{1}{12\sqrt{4n+1}}\Big[|x|\int_{-T}^T \left((x+\alpha)/{2N^2}%
\right)|f(\alpha)\alpha^3|d \alpha+\int_{-T}^T \left((x+\alpha)/{2N}%
\right)^2|f(\alpha)\alpha^4|d\alpha \\
&+\frac{1}{12N\sqrt{4n+1}}\Big[T\left(\left|\int_{-T}^T \frac{%
\sin(N\alpha)f(\alpha)\alpha^3}{x-\alpha} d \alpha\right|+\left|\int_{-T}^T
\frac{\cos(N\alpha)f(\alpha)\alpha^3}{x-\alpha} d \alpha\right|\right) \\
&+\left|\int_{-T}^T \frac{\sin(N\alpha)f(\alpha)\alpha^4}{x-\alpha} d
\alpha\right|+\left|\int_{-T}^T \frac{\cos(N\alpha)f(\alpha)\alpha^4}{%
x-\alpha} d \alpha\right|\Big].
\end{align*}


Altogether, then,
\begin{align*}
&\left[\frac{1}{2T}\int_{-T}^T \RN{3}(x)^2dx\right]^{1/2}\leq \frac{1}{48N^2%
\sqrt{4n+1}}\Big[\frac{T^3}{\sqrt 7}\int_{-T}^T |f(\alpha)\alpha^3|d \alpha+%
\frac{3T^2}{\sqrt 5}\int_{-T}^T |f(\alpha)\alpha^4|d\alpha \\
&+\frac{3T}{\sqrt{3}}\int_{-T}^T |f(\alpha)\alpha^5|d\alpha+\int_{-T}^T
|f(\alpha)\alpha^6|d\alpha\Big] +\frac{1}{6N\sqrt{4n+1}}\left[T\left(\int_{-T}^T|f(\alpha)\alpha^3|^2d%
\alpha\right)^{1/2}+\left(\int_{-T}^T|f(\alpha)\alpha^4|^2d\alpha%
\right)^{1/2}\right].
\end{align*}

\bigskip

3.3
\begin{align*}
\left|\int_{-T}^T \frac{M_3^{(n)}(x, \alpha)}{x-\alpha}f(\alpha) d
\alpha\right|&\leq \frac{1}{\sqrt{4n+3}}\left|\int_{-T}^T \frac{\alpha^3-x^3%
}{6(x- \alpha)}\cos(\sqrt{4n+3} \alpha)f(\alpha) d\alpha\right| \\
&+ \frac{|x|^3}{\sqrt{4n+3}} \left|\int_{-T}^T {sinc((x- \alpha)/2N)}%
sin(N(x+\alpha))f(\alpha) \frac{1}{12N} d \alpha\right| \\
&+\frac{|x|^3}{\sqrt{4n+3}} \left|\int_{-T}^T {sinc((x+ \alpha)/2N)}sinc
(N(x-\alpha))f(\alpha) \frac{(x+\alpha)}{12} d \alpha\right| \\
&=I(x)+II(x)+III(x).
\end{align*}

Now, $I(x)$ here is, essentially, the same as the $I(x)$ involved in $%
M_2^{(n)}(x, \alpha)$, we find
\begin{align*}
&\RN{1}(x)\leq \frac{1}{6\sqrt{4n+3}}\Big[\left|\int_{-T}^T \cos(\sqrt{4n+3}
\alpha)f(\alpha)\alpha^{2}\, d \alpha\right|+2|x|\left|\int_{-T}^T\cos(\sqrt{%
4n+3}\alpha) f(\alpha)\alpha d\alpha\right| +x^2\left|\int_{-T}^T\cos(\sqrt{4n+3}\alpha)f(\alpha)d\alpha\right|\Big] \\
&\leq \frac{1}{6{(4n+3)}}\Big[{T^2(|f(-T)|+|f(T)|)}+\int_{-T}^T{%
\left|(f^{(1)}(\alpha)\alpha^2)\right|}d\alpha+2\int_{-T}^T{|f(\alpha)\alpha|%
} \\
&+|x|\left(T\frac{(|f(-T)|+|f(T)|)}{\sqrt{4n+3}}+\int_{-T}^T\frac{%
\left|(f^{(1)}(\alpha)\alpha)\right|}{\sqrt{4n+3}}d\alpha+\int_{-T}^T\frac{%
|f(\alpha)|}{\sqrt{4n+3}}d\alpha\right) \\
&+x^2\left(\frac{|f(-T)|+|f(T)|}{\sqrt{4n+3}}+\int_{-T}^T\frac{%
|f^{(1)}(\alpha)|}{\sqrt{4n+3}}d\alpha\right)\Big].
\end{align*}
Then,
\begin{align*}
&\left[\frac{1}{2T}\int_{-T}^T \RN{1}(x)^2dx\right]^{1/2}\leq \frac{1}{6{(4n+3)}}%
\Big[(1+\frac{1}{\sqrt 3}+\frac{1}{\sqrt 5})T^2(|f(-T)|+|f(T)|) +2\int_{-T}^T|f(\alpha)\alpha|d\alpha+\frac{T}{\sqrt 3}\int_{-T}^T|f(%
\alpha)|d\alpha \\
&+\frac{T^2}{\sqrt 5}\int_{-T}^T \left|f^{(1)}(\alpha)\right| d \alpha+\frac{%
T}{\sqrt 3}\int_{-T}^T\left| f^{(1)}(\alpha)\alpha \right| d\alpha +\int_{-T}^T\left|f^{(1)}(\alpha)\alpha^2\right|d\alpha\Big].
\end{align*}
Next,
\begin{align*}
\RN{2}(x)&\leq \frac{|x|^3}{12 N\sqrt{4n+3}}\left[\int_{-T}^T|sinc((x-%
\alpha)/2N)-1||f(\alpha)|d\alpha+\left|\int_{-T}^T
\sin(N(x+\alpha)f(\alpha)d\alpha\right|\right] \\
&\leq \frac{|x|^3}{12 N\sqrt{4n+3}}\left[\int_{-T}^T\frac{1}{6} \left(\frac{%
x-\alpha}{2N}\right)^2|f(\alpha)|d\alpha+\left[\left|\int_{-T}^T {%
\cos(N\alpha)f(\alpha)}d\alpha\right|+\left|\int_{-T}^T{\sin(N\alpha)f(%
\alpha)}d\alpha\right| \right]\right] \\
&\leq \frac{|x|^3}{12N\sqrt{4n+3}}\left[\int_{-T}^T\frac{1}{24N^2}%
(x-\alpha)^2|f(\alpha)|d\alpha+\frac
2N(|f(-T)|+|f(T)|)+\frac2N\int_{-T}^T|f^{(1)}(\alpha)|d\alpha\right] \\
&+\frac{|x|^3}{288N^3\sqrt{4n+3}}\left[x^2\int_{-T}^T|f(\alpha)|d\alpha+2|x|%
\int_{-T}^T|f(\alpha)\alpha|d\alpha+\int_{-T}^T|f(\alpha)\alpha^2|d\alpha%
\right]\\
&+\frac{|x|^3}{6N^2\sqrt{4n+3}}[|f(-T)|+|f(T)|+\int_{-T}^T|f^{(1)}(%
\alpha)|d\alpha].
\end{align*}
Hence,
\begin{align*}
&\left[\frac{1}{2T}\int_{-T}^T \RN{2}(x)^2dx\right]^{1/2}\leq \frac{T^3}{6\sqrt
7 N^2\sqrt{4n+3}}\Big[|f(-T)|+|f(T)|+\int_{-T}^T|f^{(1)}(\alpha)|d\alpha%
\Big] \\
&+\frac{T^3}{288 N^3 \sqrt{4n+3}}\left[\frac{1}{\sqrt 7}\int_{-T}^T|f(\alpha)%
\alpha^2|d\alpha+\frac{2T}{\sqrt 9}\int_{-T}^T|f(\alpha)\alpha|d\alpha+\frac{%
T^2}{\sqrt{11}}\int_{-T}^T|f(\alpha)|d\alpha\right].
\end{align*}

Finally, with $T_4(t)=1-\frac{t^2}{6}+\frac{t^4}{120}$,
\begin{align*}
\RN{3}(x)&\leq \frac{|x|^3}{12\sqrt{4n+3}}\Big[\int_{-T}^T
|sinc\left((x+\alpha)/{2N}\right)-T_4((x+\alpha)/2N)|f(\alpha)|[|x|+|%
\alpha|]d \alpha \\
&+\left|\int_{-T}^T
sinc(N(x-\alpha))T_4((x+\alpha)/2N)f(\alpha)(x+\alpha)d\alpha\right|\Big] \\
&\leq \frac{|x|^3}{12\sqrt{4n+3}}\Big[\int_{-T}^T \frac{1}{5040} \left(\frac{%
x+\alpha}{2N}\right)^6|f(\alpha)|[|x|+|\alpha|] \\
&+\left|\int_{-T}^T sinc(N(x-\alpha))\left(1-\frac{(x+\alpha)^2}{24N^2}+%
\frac{(x+\alpha)^4}{1920N^4}\right)f(\alpha)(x+\alpha) d\alpha\right|\Big] \\
&\leq \frac{1}{3870720N^6\sqrt{4n+3}}\Big[x^{10}\int_{-T}^T|f(\alpha)|d%
\alpha+7|x|^9\int_{-T}^T|f(\alpha)\alpha|d\alpha+21x^8\int_{-T}^T|f(\alpha)%
\alpha^2|d\alpha \\
&+35|x|^7\int_{-T}^T|f(\alpha)\alpha^3|d\alpha+35x^6\int_{-T}^T|f(\alpha)%
\alpha^4|d\alpha+21|x|^5\int_{-T}^T|f(\alpha)\alpha^5|d\alpha \\
&+7x^4\int_{-T}^T|f(\alpha)\alpha^6|d\alpha+|x|^3\int_{-T}^T|f(\alpha)%
\alpha^7|d\alpha\Big] \\
&+\frac{1}{12N\sqrt{4n+3}}\Big[|x|^4\left(\left|\int_{-T}^T \frac{%
\sin(N\alpha)f(\alpha)}{x-\alpha} d \alpha\right|+\left|\int_{-T}^T \frac{%
\cos(N\alpha)f(\alpha)}{x-\alpha} d \alpha\right|\right) \\
&+|x|^3\left(\left|\int_{-T}^T \frac{\sin(N\alpha)f(\alpha)\alpha}{x-\alpha}
d \alpha\right|+\left|\int_{-T}^T \frac{\cos(N\alpha)f(\alpha)\alpha}{%
x-\alpha} d \alpha\right|\right)\Big] \\
&+\frac{1}{288N^2\sqrt{4n+3}}|x|^3\int_{-T}^T \left((|x|+|\alpha|)^3+\frac{%
(|x|+|\alpha|)^5}{80N^2}\right)|f(\alpha)|d\alpha \\
&=\frac{1}{3870720N^6\sqrt{4n+3}}(IV(x))+\frac{1}{12N\sqrt{4n+3}}(V(x))+%
\frac{1}{288N^2\sqrt{4n+3}}(VI(x)).
\end{align*}

Altogether, then,
\begin{align*}
\left[\frac{1}{2T}\int_{-T}^T \RN{3}(x)^2dx\right]^{1/2}&\leq \frac{1}{3870720N^6%
\sqrt{4n+3}}\Big[\frac{1}{2T}\int_{-T}^T(IV)(x)^2dx\Big]^{1/2}+\frac{1}{12N%
\sqrt{4n+3}}\left[\int_{-T}^T\frac{1}{2T}\int_{-T}^TV(x)^2dx\right]^{1/2} \\
&+ \frac{1}{288N^2\sqrt{4n+3}}\left[\frac{1}{2T}\int_{-T}^T(VI)(x)^2dx\right]%
^{1/2},
\end{align*}
where
\begin{align*}
&\left[\frac{1}{2T}\int_{-T}^T(IV)(x)^2dx\right]^{1/2} \leq\frac{T^{10}}{\sqrt{21}}\int_{-T}^T|f(\alpha)|d\alpha+\frac{7T^9}{\sqrt{%
19}}\int_{-T}^T|f(\alpha)\alpha|d\alpha \\
&+\frac{21T^8}{\sqrt{17}}\int_{-T}^T|f(\alpha)\alpha^2|d\alpha+\frac{35T^7}{%
\sqrt{15}}\int_{-T}^T|f(\alpha)\alpha^3|d\alpha+\frac{35T^6}{\sqrt{13}}%
\int_{-T}^T|f(\alpha)\alpha^4|d\alpha \\
&+\frac{21T^5}{\sqrt{11}}\int_{-T}^T|f(\alpha)\alpha^5|d\alpha +\frac{7T^4}{%
\sqrt 9}\int_{-T}^T|f(\alpha)\alpha^6|d\alpha+\frac{T^3}{\sqrt{7}}%
\int_{-T}^T|f(\alpha)\alpha^7|d\alpha.
\end{align*}

Observe that, by Bedrosian's identity,
\begin{align*}
&\int_{-T}^T \frac{\sin(N\alpha)f(\alpha)}{x-\alpha}d\alpha=\int_{-\infty}^{%
\infty}\frac{\sin(N\alpha)f(\alpha)}{x-\alpha} d\alpha-\int_{-\infty}^{%
\infty}\sin(N\alpha)f(\alpha)\xi_{|\alpha|>T}(\alpha)d\alpha \\
&=\int_{-\infty}^{\infty}\frac{\sin(N\alpha)f_N(\alpha)}{x-\alpha}
d\alpha-\int_{-\infty}^{\infty}\frac{\sin(N\alpha)[f(\alpha)-f_N(\alpha)]}{%
x-\alpha}d\alpha -\int_{-\infty}^{\infty}\sin(N\alpha)f(\alpha)\xi_{|\alpha|>T}(\alpha)d%
\alpha \\
&=\pi \cos(Nx)f_N(x)+\int_{-\infty}^{\infty}\frac{\sin(N\alpha)[f(%
\alpha)-f_N(\alpha)]}{x-\alpha} d\alpha-\int_{-\infty}^{\infty}\frac{%
\sin(N\alpha)f(\alpha)\xi_{|\alpha|>T}(\alpha)}{x-\alpha}d\alpha,
\end{align*}
where $f_N=(\hat{f}\xi_{(-N,N)})^{\vee}$. A similar result holds with $%
\sin(N\alpha)$ replaced by $\cos(N\alpha)$. Thus,
\begin{align*}
&\left[\frac{1}{2T}\int_{-T}^TV(x)^2dx\right]^{1/2}\leq \frac{1}{\sqrt{2T}}%
\left[2\pi\left|\int_{-T}^T|f_N(\alpha)\alpha^4|\right|^2d\alpha\right]%
^{1/2}+2\pi T^4\left[\int_{-\infty}^{\infty}|f(\alpha)-f_N(\alpha)|^2d\alpha%
\right]^{1/2} \\
&+2\pi T^4\left[\int_{|\alpha|>T}f(\alpha)^2dx\right]^{1/2}+2\pi T^3\left[%
\int_{-T}^T|f(\alpha)\alpha|^2d\alpha\right]^{1/2} \\
&\leq \sqrt 2 \pi \Big[T^{-1/2}\left[\int_{-T}^T|f_N(\alpha)\alpha^4|^2d%
\alpha\right]^{1/2}+T^{5/2}\left[\int_{-T}^T|f(\alpha)\alpha|^2d\alpha\right]%
^{1/2} \\
&+T^{7/2}\left[\int_{|\alpha|>T}f(\alpha)^2d\alpha\right]^{1/2}+T^{7/2}\left[%
\int_{|\omega|>N}|\hat{f}(\omega)|^2d\omega\right]^{1/2}\Big],
\end{align*}
since
\begin{align*}
\int_{-\infty}^{\infty}|f(\alpha)-f_N(\alpha)|^2d\alpha&=\int_{-\infty}^{%
\infty}|\hat{f}(\omega)-\hat{f}_N(\omega)|^2d\omega =\int_{-\infty}^{\infty}|\hat f(\omega)-\hat{f}(\omega)\xi_{(-N,N)}(%
\omega)|^2d\omega =\int_{|\omega|>N}|\hat{f}(\omega)|^2d\omega.
\end{align*}

Again,
\begin{equation*}
\Big[\frac{1}{2T}\int_{-T}^T \RN{4}(x)^2dx\Big]^{1/2}\leq \left[\frac{1}{2T}%
\int_{-T}^T\left|\int_{-T}^T|x|^3\left(8(|x|^3+|\alpha|^3)+\frac{%
2(|x|^5+|\alpha|^5)}{5N^2}\right)|f(\alpha)|d\alpha\right|^2dx\right]^{1/2}.
\end{equation*}

Expanding the integrand on the right hand side of the last inequality, we
find
\begin{align*}
&\Big[\frac{1}{2T}\int_{-T}^T \RN{4}(x)^2dx\Big]^{1/2}\leq \frac{1}{\sqrt{2T}}%
\Big[8\left[\int_{-T}^Tx^{12}dx\right]^{1/2}\left[\int_{-T}^T|f(\alpha)d%
\alpha|\right]+ 8\left[\int_{-T}^Tx^{6}dx\right]^{1/2}\left[\int_{-T}^T|f(\alpha)\alpha^3d%
\alpha|\right] \\
&+\frac{2}{5N^2}\left[\int_{-T}^Tx^{16}dx\right]^{1/2}\left[%
\int_{-T}^T|f(\alpha)|d\alpha\right] +\frac{2}{5N^2}\left[\int_{-T}^Tx^{6}dx\right]^{1/2}\int_{-T}^T|f(\alpha)%
\alpha^5|d\alpha\Big] \\
&=\frac{8}{\sqrt{13}}T^6\int_{-T}^T|f(\alpha)|d\alpha+\frac{8}{\sqrt{7}}%
T^3\int_{-T}^T|f(\alpha)\alpha^3|d\alpha+\frac{2}{5N^2\sqrt{17}}%
T^8\int_{-T}^T|f(\alpha)|d\alpha +\frac{2}{5N^2}\frac{T^3}{\sqrt{7}}\int_{-T}^T|f(\alpha)\alpha^5|d\alpha.
\end{align*}

\bigskip

3.4
\begin{align*}
 &\left|\int_{-T}^T \frac{M_4^{(n)}(x, \alpha)}{x- \alpha}f(\alpha)
d\alpha\right| \leq \frac{|x|^3}{36\sqrt{(4n+1)(4n+3)}}\Big[\left|\int_{-T}^T
\sin (N(x-\alpha))\cos((x+\alpha)/2N)\frac{f(\alpha)\alpha^3}{x-\alpha}d
\alpha \right| \\
&+\left|\int_{-T}^T \sin((x-\alpha)/2N)\cos (N(x+\alpha))\frac{%
f(\alpha)\alpha^3}{x-\alpha}d \alpha\right|\Big] \leq \frac{T^3}{18\sqrt{(4n+1)(4n+3)}}\Big[\left|\int_{-T}^T%
\frac{\cos(N\alpha)\cos(\alpha/2N)f(\alpha)\alpha^3}{x-\alpha}d \alpha
\right| \\
&+\left|\int_{-T}^T\frac{\cos(N\alpha)\sin(\alpha/2N)f(\alpha)%
\alpha^3}{x-\alpha}d \alpha \right| +\left|\int_{-T}^T\frac{\sin(N\alpha)\cos(\alpha/2N)f(\alpha)\alpha^3}{%
x-\alpha}d \alpha \right| +\left|\int_{-T}^T\frac{\sin(N\alpha)\sin(\alpha/2N)f(\alpha)%
\alpha^3}{x-\alpha}d \alpha \right|\Big].
\end{align*}

Hence,
\begin{align*}
\left[\frac{1}{2T}\int_{-T}^T\left|\int_{-T}^T \frac{M_4^{(n)}(x, \alpha)}{%
x- \alpha}f(\alpha) d\alpha\right|^2dx\right]^{1/2} &\leq \frac{\sqrt 2 \pi
T^{5/2}}{9\sqrt{(4n+1)(4n+3)}}\left[\int_{-T}^T|f(\alpha)\alpha^3|^2d\alpha%
\right]^{1/2}.
\end{align*}



\bigskip

3.5

 The expression
\begin{align*}
\left[\frac{1}{2T}\int_{-T}^T\left|\int_{-T}^T\frac{M_5^{(n)}(x, \alpha)}{%
x-\alpha}f(\alpha) d \alpha\right|^2dx\right]^{1/2}
\end{align*}
is dominated by the sum of five terms, which we now consider in turn.

\medskip

(i) The term
\begin{align*}
\left[\frac{1}{2T}\int_{-T}^T\left|\int_{-T}^T a_n {[T(2n+1, x)\cos(\sqrt{%
4n+1} \alpha)-T(2n+1, \alpha)\cos(\sqrt{4n+1} x)]} \frac{f(\alpha)}{x-\alpha}%
d \alpha \right|^2\, dx\right]^{1/2}
\end{align*}
is no bigger than
\begin{align*}
&\frac{|a_n|}{\sqrt{{2T}}}\left[\int_{-T}^T\left|T(2n+1,x)\int_{-T}^T\frac{%
\cos(\sqrt{4n+1}\alpha)f(\alpha)}{x-\alpha}d \alpha\right|^2 dx\right]^{1/2}
+\frac{|a_n|}{\sqrt{{2T}}}\left[\int_{-T}^T\left|\int_{-T}^T\frac{T(2n+1,
\alpha)f(\alpha)}{x-\alpha}d \alpha\right|^2 dx\right]^{1/2} &  \\
&\leq \frac{|a_n|\pi}{\sqrt{{2T}}}\left[%
2\int_{-T}^T|T(2n+1,x)f_N(x)|^2dx\right]^{1/2}+sup_{|x|\leq
T}|T(2n+1,x)|\left(\left(\int_{|\alpha|>T}f(\alpha)^2d\alpha\right)^{1/2}+%
\left(\int_{|\omega|>N}|\hat{f}(\omega)|^2d\omega\right)^{1/2}\right) &  \\
&\leq 4\sqrt{\frac{\pi}{2T}}\sqrt[4]{\frac 32}\frac 1n \Big[%
2\left(\int_{-T}^T|f_N(\alpha)\omega(\alpha)|^2d\alpha\right)^{1/2} +\omega(T)\Big[\left(\int_{|\alpha|>T}f(\alpha)^2d\alpha\right)^{1/2}+%
\left(\int_{|\omega|>N}|\hat{f}(\omega)|^2d\omega\right)^{1/2}\Big]\Big], & +%
\frac{|a_n|}{\sqrt{{2T}}} \left[\int_{-T}^T \left| \int_{-T}^T\frac{T(2n+1,
\alpha)f(\alpha)}{x-\alpha}d\alpha\right|^2 \, dx\right]^{1/2}
\end{align*}
in which
\begin{equation*}
\omega(\alpha)=\alpha^2\left(\frac{\alpha^4}{18}+1\right)\frac{1}{\pi^{1/2}}+%
\frac{2}{187}\frac{|\alpha|^{17/2}}{n^{1/4}}.
\end{equation*}

\medskip

(ii) Arguing as in (i) we have
\begin{align*}
&\left[\frac{1}{2T}\int_{-T}^T\left|\int_{-T}^Ta_n\left[T(2n+1,x)\frac{%
\alpha^3}{6}\frac{\sin(\sqrt{4n+1}\alpha)}{\sqrt{4n+1}}-T(2n+1, \alpha)\frac{%
x^3}{6}\frac{\sin(\sqrt{4n+1},x)}{\sqrt{4n+1}}\right]\frac{f(\alpha)}{%
x-\alpha}d \alpha\right|^2dx\right]^{1/2} \\
&\leq \frac{|a_n|}{6\sqrt{4n+1}\sqrt{2T}}\Big[\left[\int_{-T}^T%
\left|T(2n+1,x)\int_{-T}^T\frac{\sin(\sqrt{4n+1}\alpha)f(\alpha)\alpha^3}{%
x-\alpha}d\alpha\right|^2dx\right]^{1/2} +\pi T^{3}\left[\int_{-T}^T|{T(2n+1,\alpha)f(\alpha)}|^2d\alpha\right]^{1/2}%
\Big] \\
& \leq \frac{|a_n|}{6\sqrt{4n+1}\sqrt{2T}}\Big(\left[\int_{-T}^T%
\left|T(2n+1, x)\int_{-T}^T\frac{\sin(\sqrt{4n+1}\alpha)f(\alpha)\alpha^3}{%
x-\alpha}\right|^2dx\right]^{1/2}+\pi T^3\pi T^{3}\left[\int_{-T}^T|{%
T(2n+1,\alpha)f(\alpha)}|^2d\alpha\right]^{1/2}\Big) \\
&\leq \frac 23 \sqrt{\frac{\pi}{2T}}\sqrt[4]{\frac 32}\frac{1}{n\sqrt{4n+1}}%
\left[\omega(T)\left[\int_{-T}^T|f(\alpha)\alpha^3|^2d\alpha\right]^{1/2}+T^3%
\left[\int_{-T}^T|f(\alpha)\omega(\alpha)|^2d\alpha\right]^{1/2}\right].
\end{align*}

\medskip

(iii) The mean square on $I_T$ of
\begin{equation*}
\int_{-T}^Tb_n\frac{T(2n, x)\sin(\sqrt{4n+3}\alpha)-T(2n, \alpha)\sin(\sqrt{%
4n+3}x)}{x-\alpha}f(\alpha)d\alpha
\end{equation*}
is dominated by
\begin{align*}
&\frac{|b_n|}{\sqrt{2T}}\left[\left[\int_{-T}^T\left|T(2n,x)\int_{-T}^T\frac{%
\sin(\sqrt{4n+3}\alpha)f(\alpha)}{x-\alpha}d\alpha\right|^2\right]^{1/2} +%
\left[\int_{-T}^T\left|\int_{-T}^T\frac{T(2n,\alpha)f(\alpha)}{x-\alpha}%
d\alpha\right|^2\right]^{1/2}\right] \\
&\leq \frac{|b_n|\pi}{\sqrt{2T}}\left[2\left(\int_{-T}^T|T(2n,x)f_N(x)|^2%
\right)^{1/2}+sup_{|x|\leq T}|T(2n,x)|\left[\left(\int_{|\alpha|>T}f(%
\alpha)^2d\alpha\right)^{1/2}+\int_{|\omega|>N}|\hat{f}(\omega)|^2d\omega%
\right)^{1/2}\right] \\
&\leq \frac{1}{24}\sqrt{\frac{\pi^3}{2T}}\sqrt[4]{\frac 32} \frac 1n\left[2%
\left[\int_{-T}^T|f_N(\alpha)\omega(\alpha)|^2d\alpha\right]^{1/2}+\omega(T)%
\left[\left(\int_{|\alpha|>T}f(\alpha)^2d\alpha\right)^{1/2}+\left(\int_{|%
\omega|>N}|\hat{f}(\omega)|^2d\omega\right)^{1/2}\right]\right].
\end{align*}

\medskip

(iv) The method of (ii) applied to the estimation of the square mean on $I_T$%
, of
\begin{equation*}
\int_{-T}^Tb_n\left[-T(2n, \alpha)\frac{x^3}{6}\frac{\cos(\sqrt{4n+3}x)}{%
\sqrt{4n+3}}+T(2n, x)\frac{\alpha^3}{6}\frac{\cos(\sqrt{4n+3}\alpha)}{\sqrt{%
4n+3}}\right]\frac{f(\alpha)}{x-\alpha}d \alpha
\end{equation*}
leads to the upper bound
\begin{align*}
&\frac{1}{\sqrt{2T}}\frac{|b_n|}{6\sqrt{4n+3}}\sqrt[4]{\frac 1n}\left[%
\omega(T)\left[\int_{-T}^T|f(\alpha)\alpha^3|^2d\alpha\right]^2\right]%
^{1/2}+T^3\left[\int_{-T}^T|f(\alpha)\omega(\alpha)|^2d\alpha\right]^{1/2} \\
&\leq \frac{1}{24}\sqrt{\frac{\pi^3}{2T}}\sqrt[4]{\frac 32}\frac{1}{n\sqrt{%
4n+3}}\left[\omega(T)\left[\int_{-T}^T|f(\alpha)\alpha^3|^2d\alpha\right]^2%
\right]^{1/2}+T^3\left[\int_{-T}^T|f(\alpha)\omega(\alpha)|^2d\alpha\right]%
^{1/2}.
\end{align*}

\medskip

(v) The square mean, on $I_T$, of
\begin{equation*}
\int_{-T}^Ta_nb_n\left[\frac{T(2n+1, x)T(2n, \alpha)-T(2n+1, \alpha)T(2n, x)%
}{x-\alpha}\right]f(\alpha)d\alpha
\end{equation*}
is, by a now familiar argument,
\begin{align*}
&\leq \frac{|a_n||b_n|}{\sqrt{2T}}\Big[\left[\int_{-T}^T\left|T(2n+1,x)%
\int_{-T}^T\frac{T(2n, \alpha)f(\alpha)}{x-\alpha}d\alpha\right|^2dx\right]%
^{1/2} +\left[\int_{-T}^T\left|T(2n,x)\int_{-T}^T\frac{T(2n+1, \alpha)f(\alpha)}{%
x-\alpha}dx\right|^2d\alpha\right]^{1/2}\Big] \\
&\leq \frac{2\pi|a_n||b_n|}{\sqrt{2T}}\omega(T)\left[%
\int_{-T}^T|f(\alpha)\omega(\alpha)|^2d\alpha\right]^{1/2} \leq \sqrt{\frac 3T}\frac{\pi}{n^2}\omega(T)\left[\int_{-T}^T|f(\alpha)%
\omega(\alpha)|^2d\alpha\right]^{1/2}.
\end{align*}

\section{The proof of the Theorem 1}

By (2),(3), and (4) one has
\begin{equation*}
(S_Kf)(x)=-C^{(n)}\left(F_Nf_T(x)+\sum_{k=1}^5 \int_{-T}^T M_k^{(n)}(x, \alpha) \frac{f_T(\alpha)}{x-\alpha}d \alpha\right),
\end{equation*}
in which
\begin{equation*}
-C^{(n)}=\frac{1}{\pi}\left(1+\frac{\varepsilon}{12 n}\right), \quad
|\varepsilon|<3.
\end{equation*}

Thus, by Lemma 2 and (4),
\begin{align*}
&\left[\frac{1}{2T}\int_{-T}^T|f(x)-(S_{K}f)(x)|^2 dx\right]^{1/2}\leq \left[\frac{1}{2T}\int_{-T}^T
|f(x)-F_n(x)|^2 dx\right]^{1/2}  +\frac{1}{2  \pi K} \left[\frac{1}{2T}\int_{-T}^T \left|\int_{-T}^T \frac{\sin
(N(x- \alpha))}{x-\alpha}f(\alpha)d \alpha\right|^2 dx\right]^{1/2} + \frac{1}{{\pi}}\left(1+ \frac{1}{2 K}\right)S_a(K,T)\\
&\leq \left[\frac{1}{2T}\int_{|\alpha|>T}f(x)^2 dx\right]^{1/2}+\left[\frac{1}{2T}\int_{|\omega|>N}|\hat f(\omega)|^2 d\omega\right]^{1/2}+\frac{1}{K\sqrt{2T}}\Bigg(\left[\int_{|\alpha|<T}|f_N(\alpha)|^2 d\alpha\right]^{1/2}+\left[\int_{|\alpha|>T}|f(\alpha)|^2 d\alpha\right]^{1/2}\\
&+\left[\int_{|\omega|>N}|\hat f(\omega)|^2 d\omega\right]^{1/2}\Bigg)+\frac{1}{\pi}\left(1+\frac{1}{2K}\right)S_a(K,T)\\
&\leq \left(1+\frac 1K \right)\Bigg(\frac{1}{2T}\left[\int_{|\alpha|>T}f(\alpha)^2 d\alpha\right]^{1/2}+\left[\frac{1}{2T}\int_{|\omega|>N}|\hat f(\omega)|^2 d\omega\right]^{1/2}\Bigg)+\frac 1K\left[\frac{1}{2T}\int_{|\alpha|<T}|f_N(\alpha)|^2 d\alpha\right]^{1/2}+\frac{1}{\pi}\left(1+\frac{1}{2K}\right)S_a(K,T)
\end{align*}
where, once again,
\begin{align}
&S_a(K,T)=\sum_{k=1}^5 \left[\frac{1}{2T}\int_{-T}^T \left|\int_{-T}^T\frac{M_k^{(n)}(x, \alpha)}{x-\alpha}f(\alpha)d\alpha\right|^2dx\right]^{1/2}. \quad \Box
\end{align}

An explicit estimate of $S_a(K,T)$ is described in the appendix using the ones involving
 $M_k^{(n)}(x, \alpha), \, k=1, \ldots, 5$ in Section 3.

\section{An Example}

Example 1 of [KHB] involved the Hermite series approximation of the trimodal
density function

\begin{equation*}
f(t)=0.5\phi (t)+3\phi (10(t-0.8))+2\phi (10(t-1.2)),
\end{equation*}%
in which
\begin{equation*}
\phi (t)=\frac{1}{\sqrt{2\pi }}e^{-t^{2}/2},\quad t\in {\mathbb{R}},
\end{equation*}%
is the standard normal density. Figure 1 above shows $f$ is essentially
supported in $[-3,3]$. Again, from the graph of $|\hat{f}|$ in Figure 4(c)
of [KHB] we see \textit{it} effectively lives in $[-8,8]$.

\begin{center}
\begin{figure}
  \includegraphics[width=.5\textwidth]{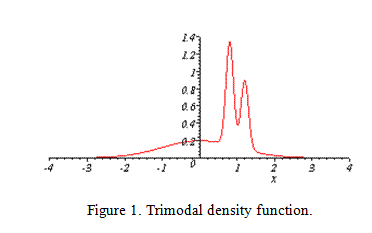}
 \end{figure}

\end{center}

Taking $T=3$ and $n=250$ (so $N=31.6544$) $K=500$, we obtain

\begin{equation}
\left[ \frac{1}{6}\int_{|t|<3}|f(t)-(S_{500}f)(t)|^{2}dt\right]
^{1/2}<0.02361.
\end{equation}

One always has
\begin{equation*}
\left[ \frac{1}{2T}\int_{|t|<T}|g(t)-(S_{K}g)(t)|^{2}dt\right] ^{1/2}\leq
sup_{|t|<T}|g(t)-(S_{K}g)(t)|,
\end{equation*}%
so, if the supremum norm is rather large, the smaller root mean square norm
gives a better measure of the average size of $|g(t)-(S_{K}g)(t)|$. In our
case
\begin{equation}
sup_{|t|<3}|f(t)-(S_{500}f)(t)|<0.0025.
\end{equation}%
Therefore, the supremum norm is \textit{here} the better measure.
Nevertheless, it is the computable estimates giving (6) that lead us  to
Figure 2 and hence to (7).

We observe that the graph in Figure 2 is of the error function $f-S_{500}f$ approximated by $f-S_{40}f-\sum_{k=41}^{500}\langle f, d_k\rangle d_k$, where $d_k$ is the Dominici approximation to $b_k$ given in Theorem 1.1 of [KHB].

\begin{center}

\begin{figure}
  \includegraphics[width=.8\linewidth]{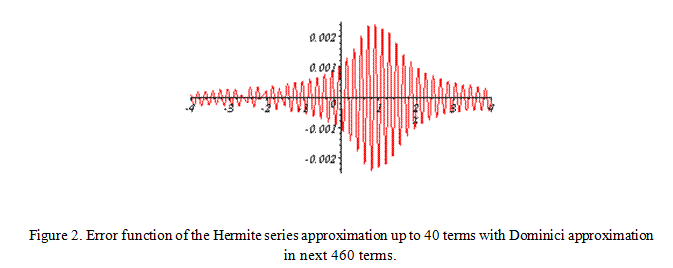}
 \end{figure}
\end{center}

\begin{figure}
  \includegraphics[width=.4\textwidth]{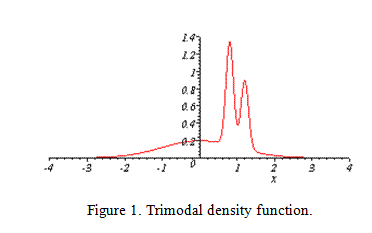}
  \centering
 \end{figure}

The term involving $S_{a}(500,3)$ in (6) makes the biggest contribution to
the upper bound in (1). Thus,
\begin{equation*}
1.002\left[ \frac{1}{6}\int_{|t|>3}f(t)^{2}dt\right] ^{1/2}<0.00051,
\end{equation*}%
\begin{equation*}
1.002\left[ \frac{1}{6}\int_{|\omega |>31.6544}|\hat{f}(\omega )|^{2}d\omega %
\right] ^{1/2}<0.00088
\end{equation*}%
and
\begin{equation*}
\frac{1}{500}\left[ \frac{1}{6}\int_{|t|<3}f(t)^{2}dt\right] ^{1/2}<0.00062,
\end{equation*}%
while
\begin{equation*}
\frac{1}{\pi }\left[ 1+\frac{1}{1000}\right] S_{a}(500,3)<0.02161.
\end{equation*}%
For the convenience of the reader we have gathered together in an appendix
the terms that make up $S_{a}(n,T)$.

\bigskip

\appendix
\begin{center}
 \textbf{Appendix}
\end{center}

\bigskip

Take the indicated multiples of the terms $\int_{-T}^{T}|f(\alpha )|d\alpha $
etc. and then add them to get an estimate  of the Sansone sum, $S_a(K,T)$ in formula (6).

\begin{align*}
\int_{-T}^{T}|f(\alpha )|d\alpha & :\frac{T^{2}}{48\sqrt{5}N^{3}}+\frac{T^{3}%
}{{384\sqrt{9}}N^{4}}+\frac{1}{8N^{3}}+\frac{T}{6\sqrt{3}\sqrt{(4n+1)(4n+3)}}%
+\frac{T}{6\sqrt{3}(4n+3)} \\
& +\frac{T^{5}}{288\sqrt{11}N^{3}\sqrt{4n+3}}+\frac{T^{10}}{3870720\sqrt{21}%
\sqrt{4n+3}N^{6}}+\frac{T^{6}}{36\sqrt{13}N^{2}\sqrt{4n+3}}+\frac{T^{8}}{720%
\sqrt{17}N^{4}\sqrt{4n+3}};
\end{align*}

\begin{align*}
\int_{-T}^{T}|f(\alpha )\alpha |d\alpha & :\frac{T}{48\sqrt{3}N^{3}}+\frac{%
3T^{2}}{{128\sqrt{5}}N^{4}}+\frac{1}{3(4n+3)}+\frac{1}{3\sqrt{(4n+1)(4n+3)}}
\\
& +\frac{T^{4}}{144\sqrt{9}N^{3}\sqrt{4n+3}}+\frac{7T^{9}}{3870720\sqrt{19}%
\sqrt{4n+3}N^{6}};
\end{align*}

\begin{align*}
\int_{-T}^{T}|f(\alpha )\alpha ^{2}|d\alpha & :\frac{1}{48N^{3}}+\frac{T}{{%
128\sqrt{3}}N^{4}}+\frac{1}{2N^{2}\sqrt{(4n+1)}}+\frac{T^{3}}{288\sqrt{%
7(4n+1)}N^{3}} \\
& +\frac{21T^{8}}{3870720\sqrt{17}\sqrt{4n+3}N^{6}};
\end{align*}

\begin{align*}
\int_{-T}^{T}|f(\alpha )\alpha ^{3}|d\alpha & :\frac{1}{384N^{4}}+\frac{T^{3}%
}{{48\sqrt{7}}N^{2}\sqrt{4n+1}}+\frac{1}{2N^{2}\sqrt{(4n+3)}}+\frac{T^{2}}{%
288\sqrt{5(4n+1)}N^{3}} \\
& +\frac{35T^{7}}{3870720\sqrt{15}\sqrt{4n+3}N^{6}}+\frac{T^{3}}{36\sqrt{7}%
N^{2}\sqrt{4n+3}};
\end{align*}

\begin{equation*}
\int_{-T}^{T}|f(\alpha )\alpha ^{4}|d\alpha :\frac{T}{144\sqrt{5}N^{3}\sqrt{%
4n+1}}+\frac{T^{2}}{{16\sqrt{5}}N^{2}\sqrt{4n+1}}+\frac{35T^{6}}{3870720%
\sqrt{13}N^{6}\sqrt{(4n+3)}};
\end{equation*}

\begin{equation*}
\int_{-T}^{T}|f(\alpha )\alpha ^{5}|d\alpha :\frac{1}{288N^{3}\sqrt{4n+1}}+%
\frac{T}{{16\sqrt{3}}N^{2}\sqrt{4n+1}}+\frac{21T^{5}}{3870720\sqrt{11}N^{6}%
\sqrt{(4n+3)}}+\frac{T^{3}}{720\sqrt{7(4n+3)}N^{4}};
\end{equation*}

\begin{equation*}
\int_{-T}^{T}|f(\alpha )\alpha ^{6}|d\alpha :\frac{1}{48N^{2}\sqrt{4n+1}}+%
\frac{7T^{4}}{3870720\sqrt{9}\sqrt{4n+3}N^{6}};
\end{equation*}

\begin{equation*}
\int_{-T}^{T}|f(\alpha )\alpha ^{7}|d\alpha :\frac{T^{3}}{3870720\sqrt{7}%
\sqrt{4n+3}N^{6}};
\end{equation*}

\begin{equation*}
\int_{-T}^{T}|f^{(1)}(\alpha )|d\alpha :\frac{1}{2N^{2}}+\frac{1}{4N^{3}}+%
\frac{T^{2}}{6\sqrt{5}\sqrt{(4n+1)(4n+3)}}+\frac{T^{2}}{6\sqrt{5}(4n+3)};
\end{equation*}

\begin{equation*}
\int_{-T}^{T}|f^{(1)}(\alpha )\alpha |d\alpha :\frac{T}{6\sqrt{3}\sqrt{%
(4n+1)(4n+3)}}+\frac{T}{6\sqrt{3}(4n+3)}+\frac{T^{3}}{6\sqrt{7}N^{2}\sqrt{%
4n+3}};
\end{equation*}

\begin{equation*}
\int_{-T}^{T}|f^{(1)}(\alpha )\alpha ^{2}|d\alpha :\frac{1}{6\sqrt{%
(4n+1)(4n+3)}}+\frac{1}{6(4n+3)};
\end{equation*}

\begin{equation*}
\int_{-T}^{T}|f^{(1)}(\alpha )\alpha ^{3}|d\alpha :\frac{1}{6\sqrt{(4n+1)}%
N^{2}};
\end{equation*}

\begin{equation*}
\omega (\alpha )=\alpha ^{2}\left( \frac{\alpha ^{4}}{18}+1\right) \frac{1}{%
\pi ^{1/2}}+\frac{2}{187}\frac{|\alpha |^{17/2}}{n^{1/4}},
\end{equation*}%

\begin{equation*}
\left[ \int_{-T}^{T}|f(\alpha )\alpha |^{2}d\alpha \right] ^{1/2}:\frac{\pi
T^{5/2}}{6\sqrt{2}N\sqrt{4n+3}}+\frac{\pi T^{1/2}}{4N^2};
\end{equation*}%

\begin{equation*}
\left[ \int_{-T}^{T}|f(\alpha )\alpha^{2} |^{2}d\alpha \right] ^{1/2}:\frac{\sqrt 2 \pi
}{8N^2}T^{-1/2};
\end{equation*}%

\begin{align*}
\left[ \int_{-T}^{T}|f^{{}}(\alpha )\alpha ^{3}|^{2}d\alpha \right] ^{1/2}& :%
\frac{\pi T^{1/2}}{6\sqrt{2(4n+1)}N}+\frac{\sqrt{2}\pi T^{5/2}}{9\sqrt{%
(4n+1)(4n+3)}} \\
& +\frac{2}{3}\sqrt{\frac{\pi }{2}}\left( \frac{3}{2}\right) ^{1/4}\frac{%
T^{-1/2}\omega (T)}{n\sqrt{4n+1}}+\frac{1}{24}\sqrt{\frac{\pi ^{3}}{2}}%
\left( \frac{3}{2}\right) ^{1/4}\frac{T^{-1/2}\omega (T)}{n\sqrt{4n+3}};
\end{align*}

\begin{equation*}
\left[ \int_{-T}^{T}f_N(\alpha )^{2}d\alpha \right] ^{1/2}:\frac{\sqrt 2 \pi
}{8N^2}T^{3/2};
\end{equation*}%

\begin{equation*}
\left[ \int_{-T}^{T}|f_{N}(\alpha )\alpha ^{4}|^{2}d\alpha \right] ^{1/2}:%
\frac{\sqrt{2}\pi T^{-1/2}}{12N\sqrt{4n+3}},
\end{equation*}%
where $f_{N}=(\hat{f}\xi _{(-N,N)})^{\vee }=\frac{1}{\pi }\int_{-\infty
}^{\infty }\frac{\sin N(t-s)}{t-s}f(s)ds$;


\begin{align*}
\left[ \int_{-T}^{T}|f(\alpha )\omega (\alpha )|^{2}d\alpha \right] ^{1/2}& :%
\frac{\sqrt{2}}{3}\pi ^{1/2}T^{5/2}\left( \frac{3}{2}\right) ^{1/4}\frac{1}{n%
\sqrt{4n+1}} \\
& +\frac{1}{24\sqrt{2}}\pi ^{3/2}T^{5/2}\left( \frac{3}{2}\right) ^{1/4}%
\frac{1}{n\sqrt{4n+3}}+\sqrt{2}\pi T^{-1/2}\frac{1}{n^{2}}\omega (T);
\end{align*}

\begin{equation*}
\left[ \int_{-T}^{T}|f_{N}(\alpha )\omega (\alpha )|^{2}d\alpha \right]
^{1/2}:4\sqrt{2}\pi ^{1/2}T^{-1/2}\left( \frac{3}{2}\right) ^{1/4}\frac{1}{n}%
+\frac{1}{12\sqrt{2}}\pi ^{3/2}T^{5/2}\left( \frac{3}{2}\right) ^{1/4}\frac{1%
}{n};
\end{equation*}

\begin{align*}
|f(-T)|+|f(T)|& :\frac{1}{2N^{2}}+\frac{T}{8N^{3}}\left( 1+\frac{1}{\sqrt{3}}%
\right) +\frac{T^{2}}{6\sqrt{(4n+1)(4n+3)}}\left( 1+\frac{1}{\sqrt{3}}+\frac{%
1}{\sqrt{5}}\right)  \\
& +\frac{T^{3}}{6N^{2}\sqrt{4n+1}}+\frac{T^{2}}{6{(4n+3)}}\left( 1+\frac{1}{%
\sqrt{3}}+\frac{1}{\sqrt{5}}\right) +\frac{T^{3}}{6\sqrt{7}N^{2}\sqrt{4n+3}};
\end{align*}

\begin{align*}
\left[ \int_{|\alpha |>T}|f(\alpha )^{2}|d\alpha \right] ^{1/2}& :2\sqrt{2}%
\pi ^{1/2}T^{-1/2}\left( \frac{3}{2}\right) ^{1/4}\frac{\omega (T)}{n}+\frac{%
\sqrt{2}\pi T^{7/2}}{12N\sqrt{4n+3}} \\
& +\frac{1}{24\sqrt{2}}\pi ^{3/2}T^{-1/2}\left( \frac{3}{2}\right) ^{1/4}%
\frac{\omega (T)}{n};
\end{align*}

\begin{align*}
\left[ \int_{|\omega |>N}|\hat{f}(\omega )^{2}|d\omega \right] ^{1/2}& :2%
\sqrt{2}\pi ^{1/2}T^{-1/2}\left( \frac{3}{2}\right) ^{1/4}\frac{\omega (T)}{n%
}+\frac{\sqrt{2}\pi T^{7/2}}{12N\sqrt{4n+3}} \\
& +\frac{1}{24\sqrt{2}}\pi ^{3/2}T^{-1/2}\left( \frac{3}{2}\right) ^{1/4}%
\frac{\omega (T)}{n}.
\end{align*}%
%
%

\bigskip

\begin{center}
\textbf{Acknowledgment}

\medskip
\end{center}

\emph{The research was supported by the Natural Sciences and Engineering
Council of Canada (NSERC) grant MLH, RGPIN-2014-04621.}


\begin{thebibliography}{HKB}

\bibitem[B]{Bedrosian}

{E.~Bedrosian} \emph{\ A product theorem for Hilbert transform,} Proc. IEEE
51 (1959), 868--869.

\bibitem[KHB]{KHB}

{R.~Kerman, M.L.~Huang, M.~Brannan} \emph{\ Error estimates for Dominici's
Hermite function asymptotic formula and some applications,} The ANZIAM
Journal 50 (2009), 550--561.

\bibitem[S]{Uspensky:1927}

{G.~Sansone} \emph{\ Orthogonal Functions,} Pure and Applied Math.,
Interscience Publishers, Inc., New York, 1959.


\bibitem[W]{Wiener}

{N. Wiener} \emph{\ The Fourier Integral and certain of its Applications,}
Cambridge University Press, New York, 1933.
\end{thebibliography}
\end{document}